\documentclass[12pt]{article}
\usepackage{amsfonts}
\usepackage{amsmath,color,amsthm,amssymb,cite,bbm,mathrsfs}
\usepackage{epsfig}
\usepackage{graphicx}
\usepackage{latexsym}

\usepackage{color}
\title{The Location of the First Ascent in a 123-Avoiding Permutation}
\author {Samuel Connolly\\ Department of Mathematics\\ University of Pennsylvania\and Zachary Gabor\\ Department of Mathematics\\Haverford College\and Anant P.~Godbole\\
Department of Mathematics and Statistics\\
East Tennessee State University}
\begin{document}
\def\qed{\vbox{\hrule\hbox{\vrule\kern3pt\vbox{\kern6pt}\kern3pt\vrule}\hrule}}
\def\ms{\medskip}
\def\n{\noindent}
\def\ep{\varepsilon}
\def\G{\Gamma}
\def\lr{\left(}
\def\ls{\left[}
\def\rs{\right]}
\def\lf{\lfloor}
\def\rf{\rfloor}
\def\lg{{\rm lg}}
\def\lc{\left\{}
\def\rc{\right\}}
\def\rr{\right)}
\def\ph{\varphi}
\def\p{\mathbb P}
\def\nk{n \choose k}
\def\cA{{\cal A}}
\def\s{\cal S}
\def\e{\mathbb E}
\def\v{\mathbb V}
\newcommand{\begp}{\begin{proposition}}
\newcommand{\enp}{\end{proposition}}
\newcommand{\begt}{\begin{thm}}
\newcommand{\ent}{\end{thm}}
\newcommand{\begl}{\begin{lem}}
\newcommand{\enl}{\end{lem}}
\newcommand{\begc}{\begin{corollary}}
\newcommand{\enc}{\end{corollary}}
\newcommand{\begcl}{\begin{claim}}
\newcommand{\encl}{\end{claim}}
\newcommand{\begr}{\begin{remark}}
\newcommand{\enr}{\end{remark}}
\newcommand{\begal}{\begin{algorithm}}
\newcommand{\enal}{\end{algorithm}}
\newcommand{\begd}{\begin{definition}}
\newcommand{\enf}{\end{definition}}
\newcommand{\begx}{\begin{example}}
\newcommand{\enx}{\end{example}}
\newcommand{\bega}{\begin{array}}
\newcommand{\ena}{\end{array}}
\newcommand{\bsno}{\bigskip\noindent}
\newcommand{\msno}{\medskip\noindent}
\newcommand{\oM}{M}
\newcommand{\omni}{\omega(k,a)}
\newcommand{\yjn}{Y_{j,n}}
\newcommand\gd{\delta}
\newcommand\gD{\Delta}
\newcommand\gl{\lambda}
\newcommand\gL{\Lambda}
\newcommand{\ijn}{I_{j,n}}
\newcommand{\Bin}{B_{i,n}}
\newcommand{\yjkn}{Y_{j,k,n}}
\newcommand{\ijkn}{I_{j,k,n}}
\newcommand{\Bikn}{B_{i,k,n}}
\newcommand{\skan}{S_k(\alpha,n)}
\newcommand{\san}{S_2(\alpha,n)}
\hyphenation{Quick-sort}
\newcommand\urladdrx[1]{{\urladdr{\def~{{\tiny$\sim$}}#1}}}
\newtheorem{thm}{Theorem}[section]
\newtheorem{con}{Conjecture}[section]
\newtheorem{claim}[thm]{Claim}
\newtheorem{definition}[thm]{Definition}
\newtheorem{lem}[thm]{Lemma}
\newtheorem{cor}[thm]{Corollary}
\newtheorem{remark}[thm]{Remark}
\newtheorem{prp}[thm]{Proposition}
\newtheorem{ex}[thm]{Example}
\newtheorem{eq}[thm]{equation}
\newtheorem{que}{Problem}[section]
\newtheorem{ques}[thm]{Question}
\providecommand{\floor}[1]{\left\lfloor#1\right\rfloor}
\maketitle
\begin{abstract}  
It is natural to ask, given a permutation with no three-term ascending subsequence, at what index the first ascent occurs.  We shall show, using both a recursion and a bijection, that the number of 123-avoiding permutations at which the first ascent occurs at positions $k,k+1$ is given by the $k$-fold Catalan convolution $C_{n,k}$ \cite{andr},\cite{rege},\cite{telf}.  For $1\le k\le n$, $C_{n,k}$ is also seen to enumerate the number of 123-avoiding permutations with $n$ being in the $k$th position.  Two interesting discrete probability distributions, related obliquely to the Poisson and geometric random variables, are derived as a result.
\end{abstract}
\section{Introduction} For $n\ge0$, the Catalan numbers $C_n$ are given by
\[C_n=\frac{1}{n+1}{{2n}\choose{n}};\] generalizing this fact, 
Catalan \cite{cata} proved the $k$-fold Catalan convolution formula
\[C_{n,k}:=\sum_{i_1+\ldots+i_k=n}\prod_{r=1}^kC_{i_r-1}=\frac{k}{2n-k}{{2n-k}\choose{n}}.\]
The theory of pattern avoidance in permutations is now well-established and thriving, and a survey of the many results in that area may be found in the text by Kitaev\cite{kita}.  One of the earliest and most fundamental results in the field is that the number of permutations of $[n]:=\{1,2,\ldots,n\}$ in which the longest increasing sequence is of length $\le2$, the so-called 123-avoiding permutations, is given by the Catalan numbers, and classical bijective techniques give that the each of the $ijk$-avoiding permutations with $\{i,j,k\}=\{1,2,3\}$ are equinumerous.  In this paper, we ask a very natural question, namely in how many permutations in which the longest increasing subsequence is of length at most 2, does the first ascent occur in positions $k, k+1$.  Actually, we were web-searching for the answer to this question to use in a different context, and were rather surprised to find that the solution appeared to be not known explicitly:  Bousquet-M\'elou \cite{bous} addressed this question indirectly when she used the location of the first ascent as a ``catalytical variable" in the ``laziest proof, combinatorially speaking," of the fact that there are $C_n$ 123-avoiding permutations.  In Sections 2 and 3 of this paper, we will give recursive and bijective proofs respectively of the fact that there are $C_{n,k}$ 123-avoiding permutations on $[n]$ for which the first ascent occurs at positions $k,k+1$.  Critical to the bijective proof are the various combinatorial interpretations of the Catalan convolutions due to Tedford \cite{telf}, and the bijections between Dyck paths and avoiding permutations due to Krattenthaler\cite{krat}.  In Section 2, additionally, we show that $C_{n,k}$ also enumerates 123-avoiding permutations with the position of ``$n$" being $k\in[1,n]$.  Finally, in Section 4, we produce two interesting probability distributions on ${\bf Z}^+$ related to these issues. These are reminiscent of the geometric and Poisson random variables, and are studied systematically in \cite{jie}.
\section{Recursive Proof}

Throughout, we refer to the first ascent as being in position $k$ if the ascent is at the $k$th and $k+1$st positions of the permutation. If a permutation of $[n]$ letters has no ascents at all (i.e., it is the permutation $\{n,n-1,\ldots,2,1\}$), we define it as having first ascent in position $n$, as though it had a first ascent ``after" the last letter in the permutation -- perhaps using a number such as $1.5$ in the $(n+1)$st spot.

Any $123$-avoiding permutation of $[n]$ becomes a $123$-avoiding permutation of $[n-1]$ when the letter $n$ is removed from it, so we can view each $123$-avoiding permutation of $[n]$ as being grown uniquely by taking a $123$-avoiding permutation of $[n-1]$ and inserting the letter $n$ into certain positions.

Suppose we have a $123$-avoiding permutation of $[n-1]$ with first ascent in position $k$. How can this be grown into a longer 123-avoiding permutation by inserting $n$?  If $n$ is inserted at the very front of this permutation,  the resulting permutation is still $123$-avoiding with the original first ascent being ``pushed forward" one position due to the presence of $n$.  If $n$ is inserted anywhere between the very front of this permutation and the peak of the original first ascent, i.e., as anything from the $2$nd letter to the $k+1$st letter of the permutation, then $n$ becomes the peak of the new first ascent, which therefore has position $1$ less than the position of $n$. Furthermore, as there are no ascents before $n$, and as $n$ cannot be involved in any subsequent ascents, the new permutation is still $123$-avoiding. 
Finally, if $n$ is inserted after the original first ascent, then the resulting permutation is no longer $123$-avoiding.
Note that all of the above hold exactly even if the original permutation is the descending permutation $\{n-1,n-2,\ldots,2,1\}$ with first ascent in position $k=n-1$.  To summarize, a $123$-avoiding permutation of length $n-1$ with first ascent in position $k$ gives rise to exactly one $123$-avoiding permutation of length $n$ with first ascent in position $i$ for each $1\leq i\leq k+1$, and each $123$-avoiding permutation of length $n$ must be grown in such a way.

Turning this around, the number $A_{n,k}$ of $123$-avoiding permutations of length $n$ with first ascent in position $k$ is equal to the number of $123$-avoiding permutations of length $n-1$ with first ascent in position $k-1$ or later, as it is these permutations that give rise to them. Thus

$$A_{n,k}=\displaystyle\sum_{i=k-1}^n A_{n-1,i}$$

In particular, we note that 

$$A_{n,k}=\displaystyle\sum_{i=k-1}^n A_{n-1,i}=A_{n-1,k-1}+\displaystyle\sum_{i=k}^n A_{n-1,i}=A_{n-1,k-1}+A_{n,k+1}.$$

Since $A_{n,0}=0$ for all $n$, the above recursion indicates that $A_{n,1}=A_{n,2}$ for all $n$, with both being equal to the total number of $123$-avoiding permutations of length $n-1$. We can see that this is true: any such $n-1$ permutation can be uniquely grown into a $123$-avoiding permutation with first ascent in position 2 by inserting $n$ either as the third letter (if the original permutation did not have first ascent in position $1$) or as the first letter (otherwise); or it can be uniquely grown into a $123$-avoiding permutation with first ascent in position $1$ by inserting $n$ as the second letter.
Combined with the base cases $A_{n,0}=0$ and $A_{n,n}=1$ for all $n\geq 1$, this recurrence relation is sufficient to fully characterize $A_{n,k}$ for any $1\leq k\leq n$.

Note that $C_{n,0}=0$ and $C_{n,n}=1$ for all $n\geq 1$. Moreover, we find that the Catalan convolutions $C_{n,k}$ obey the same recurrence relation as above:

\begin{eqnarray*}&&C_{n-1,k-1}+C_{n,k+1}\\
\quad&{}&=\frac{k-1}{2n-k-1}\binom{2n-k-1}{n-1} + \frac{k+1}{2n-k-1}\binom{2n-k-1}{n}\\
\quad&{}&=\frac{k-1}{2n-k-1}\left(\frac{n}{2n-k}\right)\binom{2n-k}{n} + \frac{k+1}{2n-k-1}\left(\frac{n-k}{2n-k}\right)\binom{2n-k}{n}\\
\quad&{}&=\frac{k}{2n-k}\binom{2n-k}{n}=C_{n,k}.\end{eqnarray*}

Since $C_{n,k}$ obey the same recurrence relation as the $A_{n,k}$, and they have the same base cases (which generate their values for all $1\leq k\leq n$), we find that $C_{n,k}=A_{n,k}$ everywhere.

\begin{cor} The number of 123-avoiding permutations where $n$ is in the $k$th spot are also given by the Catalan convolutions $C_{n,k}$.  \end{cor}
\begin{proof}
The result is obvious for $k=1$ where the answer equals $C_{n-1}=C_{n,1}$.  Let $k=2$.  We have that $n$ is in position 2 in a 123-avoiding permutation iff the first ascent is at positions (1,2), necessarily to $n$.  Thus again there are $C_{n,1}=C_{n-1}$ such possibilities.  For $k\ge3$ let $\alpha_{n,k}$ be the number of 123-avoiding permutations on $[n]$ with $n$ in the $k$th spot.  Since, as will be emphasized in Section 3, for the first ascent to be at spots $(k-1,k)$, $n$ must either be in position 1, or the first ascent must be to $n$, we have that
\[\alpha_{n,k}=C_{n,k-1}-C_{n-1,k-2}=C_{n,k},\]
by the above recursion.

To give an alternate bijective proof for $k\ge 2$, we proceed as follows.  Consider a 123-avoiding permutation $\pi$ with first ascent at spots $(k,k+1)$,  and move $n$, originally in position 1 or $k+1$, into position $k$, while keeping the relative order of the other numbers unchanged.  Regardless of whether $n$ was in position 1 or position $(k+1)$, the new permutation has first ascent at position $k-1$ and is still 123-free.  Since just one of these original configurations is valid for a given relative ordering of $[n-1]$, we see that this map $\varphi$ from the set of 123-avoiders with first ascent at $(k, k+1)$ to the set of 123-avoiding permutations with $n$ in position $k$ is one-to-one.  Moreover the map has an inverse: If $\varphi(\pi)(k-1)< \varphi(\pi)(k+1)$, $n$ must have been at the beginning of $\pi$, and $n$ must have been in position $k+1$ if we find that $\varphi(\pi)(k-1)> \varphi(\pi)(k+1)$. \hfill\end{proof}

\section{Bijective Proof}
We have, from Tedford \cite{telf} that $C_{n,k}$ is given (adjusting for his different indexing) by the number of lattice paths from $(k-1, 0)$ to $(n-1, n-1)$ consisting of steps of $(0,1)$ and $(1,0)$ and never crossing above the line $x=y$. Note that these paths, and hence, $C_{n,k}$ are in bijection with these same types of paths between $(k, 1)$ and $(n,n)$. We will show that these paths are in bijection with the paths corresponding, by the Krattenthaler bijection, to $123$-avoiding $n$-permutations with first ascent at $k$.

Krattenthaler's bijection between $123$-avoiding permutations and Dyck paths can be described as follows \cite{krat}: Given a permutation $\pi$ of $n$ integers, denote the right to left maxima (RLM) of $\pi$, reading from left to right, by $\{m_s, m_{s-1}, \dots , m_2, m_1\}$ and denote the (necessarily descending) word between $m_{i+1}$ and $m_i$ by $w_i$. $\pi$ will now read left to right as $w_sm_s \dots w_2m_2w_1m_1$. Read $\pi$ from left to right, and draw as follows, beginning at $(0,0)$: upon encountering $w_i$ add $|w_i|+1$ steps in the $x$ direction, where $|w_i|$ is the length of $w_i$. Upon encountering $m_i$, add $m_i-m_{i-1}$ steps in the $y$ direction ($m_0$ is taken to be $0$ by convention). This will give a lattice path between $(0,0)$ and $(n,n)$ consisting of steps of $(1,0)$ and $(0,1)$ and never crossing above the line $x=y$.
\begin{ex}
The permutation $76584213$ corresponds with the path encoded by $XXXXYYYYXYXXXYYY$ where $X, Y$ represent steps to the east ($E$) and north ($N$) respectively.
\end{ex}
\begin{lem}
If $\pi$ is a 123-avoiding $n$-permutation with first ascent at $k$, then $\mu$, the leftmost right to left maximum preceded by a non-empty word, is at position $k+1$.  Also, $\mu$ is either $n$, or $\mu$ is one less than the nearest right to left maximum to its left.
\end{lem}
\begin{proof}
The first ascent of $\pi$ must be to a RLM, as follows: If the first ascent is either at positions $n-1$ or $n$, we easily or vacuously have the next symbol being a RLM.  If the first ascent is at position $k\in[1,n-2]$ then the $k+1$st symbol must be a RLM, since otherwise the next RLM to the right would enable the formation of a 123. If the permutation does not begin with $n$, then $n$ must be at the top of the first ascent. This is because if $n$ is not the first number, it must be the top of an ascent, but if an ascent precedes that with the $n$, that ascent, along with $n$, would form a $123$. This is the case in which $\mu=n$. If $n$ is the first term in $\pi$, then $\pi$ begins with the integers $n(n-1)...(n-a)$ for some $0 \leq a \leq k-2$ (we choose the maximum such $a$, and from now on we will say that ``$\pi$ begins with a regular descent from $n$ of length $a+1$"). The upper bound on $a$ comes from the fact that if $\pi$ began with regular descent from $n$ of length $k$ (i.e., if we had $a=k-1$) then all integers greater than the one appearing at index $k$ have already appeared and so the first ascent could not be at $k$. After the end of the regular descent from $n$, we can think of the rest of the permutation as a $123$-avoiding $(n-a-1)$-permutation, and so, by the same reasoning as above, along with the fact that $n-a-1$ cannot be the first term, or else it would lengthen the regular descent from $n$, $n-a-1$ must be at the top of the first ascent, and must be preceded by a non-empty word.
\end{proof}
\begin{thm}
The Dyck paths that correspond, by the Krattenthaler bijection, with $123$-avoiding $n$-perms with first ascent at $k$ are in bijection with the lattice paths given by Tedford \cite{telf}, as counted by $C_{n,k}$.
\end{thm}
\begin{proof}
Let $\pi$ be a $123$-avoiding $n$-permutation with first ascent at $k$. By Lemma 3.2, the following two cases are exhaustive: either (i) the first ascent is to $n$, and so $|w_s|=k$ meaning that the path begins with $k+1$ $x$ steps for a word of length $k$, followed by a $y$ step for an RLM, or (ii) $\pi$ begins with a regular descent from $n$ of length $j$ where $1 \leq j \leq k-1$ and subsequently contains a word of length $k-j$, and then a RLM with value less than the last term in the regular descent from $n$. In the latter case, the path will begin with $j$ iterations of the pattern ($x$ step, $y$ step) each representing an empty word followed by a RLM one greater than the following RLM, and then will have $k-j+1$ $x$ steps, corresponding to the word of length $k-j$, and then a $y$ step corresponding to a RLM. In the first case, we have a specific path from $(0, 0)$ to $(k+1, 1)$. In the second case, we have one specific path for each $1 \leq j \leq k-1$ from $(0,0)$ to $(k+1, j+1)$. This is to say that every Dyck path that gets to one of these points via the path associated with it represents a $123$-avoiding permutation with first ascent at $k$, and vice versa. Therefore the number of $123$-avoiding $n$-permutations with first ascent at $k$ is given by the number of unique ways to finish a Dyck path from each of these endpoints, i.e., denoting by  ``good" paths the ones that do not cross the line $x=y$,
$$C_{n,k}=\sum_{i=1}^{k}\vert\text{good lattice paths with $E/N$ steps from $(k+1,i)$ to $(n,n)$}\vert.$$
A bijection between these paths and the paths from $(k,1)$ to $(n,n)$ is given by taking a path from $(k,1)$ to $(n,n)$, and disregarding every step up through the first $x$ step, so that a path from $(k+1,i);  1 \leq i \leq k$ to $(n,n)$ remains.
\end{proof}
\section{Limit Distributions}  The probability that a random permutation on $[n]$ has its first ascent at position $k$ is given, for $1\le k\le n-1$, by $\frac{k}{(k+1)!}$.  To see this, choose any one of the $k+1$ elements in positions 1 through $k+1$, except for the smallest, to occupy the $k+1$st position, and then arrange the other elements in a monotone decreasing fashion.  The chance that the first ascent is at position $n$ is, of course, $\frac{1}{n!}$.  We will find it more convenient in this section to consider infinite analogs of the finite distributions we derive.  An infinite permutation may be realized, e.g., by considering the order statistics $X_{(1)}<X_{(2)}<\ldots$ of a sequence $X_1,X_2\ldots$ of independent and identically distributed (i.i.d.) random variables with say a uniform distribution on [0,1].  Under this scheme we get the first ascent distribution as being
\[f(x)=\frac{x}{(x+1)!}, x=1,2,\ldots,\]
which is similar in form to the unit Poisson distribution with parameter $\lambda=1$ -- and mass function $g(y)=e^{-1}/y!; y=0,1,\ldots$, mean and variance equal to 1,  and generating function $\e(s^Y)=\exp\{s-1\}$.  By contrast, it is shown in \cite{jie} that the first ascent distribution above satisfies 
\[\e(X)=e-1; \v(X)=e(3-e); \e(s^X)=\frac{(1-e^s+se^s)}{s}.\]

What, on the other hand, can be said about the location distribution of the first ascent in a random 123-avoiding permutation?  We see from our earlier results that for a randomly chosen 123-avoiding permutation on $[n]$ the distribution of the location of first ascent is given by
\[f(k)=\frac{C_{n,k}}{C_n}=k\frac{(2n-k-1)!(n+1)!}{(2n)!(n-k)!},\enspace k=1,2,\ldots,n,\]
which, for small $k$ and large $n$, may be approximated by
$f(k)=\frac{k}{2^{k+1}}.$
Accordingly, in \cite{jie} the authors studied the geometric-like distribution on ${\bf Z}^+=1,2,\ldots$ defined by
\[f(w)=\frac{w}{2^{w+1}}, w=1,2,\ldots,\]
showing that 
\[\e(W)=3; \e(s^W)=\frac{s}{s^2-4s+4}.\]  Roughly speaking, the above facts indicate that for a random permutation on a large $[n]$, we expect the first ascent to be at position $e-1\approx 1.718$, whereas this value increases to 3 for a random 123-avoiding permutation.

\section{Open Questions}  A whole series of questions would relate to enumeration of permutations, free of a certain pattern, in which the first occurrence of another pattern occurs at a certain spot.  Another direction to pursue might be to consider a specific partial order on $[n]$ and answer the same question as that studied in this paper.  Finally, can we use the notion of first ascents in the context of 123-avoiding permutations to give another combinatorial proof of Shapiro's Catalan Convolution identity, as in \cite{andr}, \cite{nagy} (both papers were written in response to a query of R.~M.~Stanley)?

\section{Acknowledgments} The research of all three authors was supported by NSF Grant 1263009.


\begin{thebibliography}{99}
\bibitem{andr} G.~Andrews (2011).  ``On Shapiro's Catalan convolution," {\it Adv. Appl. Math.} {\bf 46}, 15--24.
\bibitem{bous} M.~Bousquet-M\'elou (2011).  ``Counting permutations with no long monotone subsequence via generating trees and the kernel method," {\it J. Alg. Combin.} {\bf 33}, 571--608.
\bibitem{cata} E.~Catalan (1887). ``Sur les nombres de Segner," {\it Rend. Circ. Mat. Palermo} {\bf 1}, 190--201.
\bibitem{jie} A.~Godbole and J.~Hao (2014+).  ``Two probability distributions emanating from permutation patterns," submitted.
\bibitem{kita} S.~Kitaev (2011).  {\it Patterns in Permutations and Words,} Springer Verlag, Berlin.
\bibitem{nagy} G.~Nagy (2012).  ``A combinatorial proof of Shapiro's Catalan convolution," {\it Adv. Appl. Math.} {\bf 49}, 391--396.
\bibitem{krat} C.~Krattenthaler (2001).  ``Permutations with restricted patterns and Dyck paths," {\it Adv. Appl. Math.} {\bf 27}, 510--530.
\bibitem{rege} A.~Regev (2012). ``A proof of Catalan's convolution formula," {\it Integers:  Journal of Combinatorial Number Theory} {\bf 12}, Paper \#A29, 6 pages.
\bibitem{telf} S.~Tedford (2011).  ``Combinatorial interpretations of convolutions of the Catalan numbers," {\it Integers:  Journal of Combinatorial Number Theory} {\bf 11}, Paper \#A3, 10 pages.
\end{thebibliography}
\end{document}